\documentclass[11pt,reqno]{amsart}
  \usepackage{hyperref}
  \usepackage{amsmath,amssymb}
  \usepackage{xcolor}
  \usepackage{lineno}
  \usepackage{calc}
  \usepackage{amsthm}
  \usepackage{centernot}
  
  \theoremstyle{plain}

  \newtheorem{thm}{Theorem}[section]
  \newtheorem{lemma}[thm]{Lemma}

  \theoremstyle{definition}
  \theoremstyle{nonumberplain}
  
  \newtheorem{defn}[thm]{Definition}
  \newtheorem{example}[thm]{Example}

  \newtheorem{rem}[thm]{Remark}
  
  \numberwithin{equation}{section}

  \newcommand{\interior}[1]{{\kern0pt#1}^{\mathrm{o}}}

\begin{document}
  	
  	
  	\title[On nonempty intersection properties in metric spaces]{On nonempty intersection properties in metric spaces}
  	
  	\author[A. Gupta]{Ajit Kumar Gupta $^\dagger$}
  	
  	\address{$^\dagger$Department of Mathematics\\ NIT Meghalaya\\ Shillong 793003\\ India}
  	\email{ajitkumar.gupta@nitm.ac.in}
  	\author[S. Mukherjee]{Saikat Mukherjee  {$^{\dagger *}$}}
  	\address{$^\dagger$Department of Mathematics\\ NIT Meghalaya\\ Shillong 793003\\ India}
  	\email{saikat.mukherjee@nitm.ac.in}

  	$\thanks{*Corresponding author}$
  	
  	
  	\begin{abstract}  		
  	The classical Cantor's intersection theorem states that in a complete metric space $X$, intersection of every decreasing sequence of nonempty closed bounded subsets, with diameter approaches zero, has exactly one point. In this article, we deal with decreasing sequences $\{K_n\}$ of nonempty closed bounded subsets of a metric space $X$, for which the Hausdorff distance $H(K_n, K_{n+1})$ tends to $0$, as well as for which the excess of $K_n$ over $X\setminus K_n$ tends to $0$. We achieve nonempty intersection properties in metric spaces. The obtained results also provide partial generalizations of Cantor's theorem.

  	\end{abstract}
  	
  	\maketitle

\section{Introduction}\label{Intro}
 	In metric spaces, there are some marvellous nonempty intersection theorems. Cantor's theorem (see \cite{gf63}) asserts that in a complete metric space $X$, the intersection of every decreasing sequence $\{K_n\}$ of nonempty closed subsets of $X$ with diameter $\delta(K_n)\to 0$ has exactly one point. This intersection theorem is widely used in the fields related to mathematical analysis. Kuratowski \cite{kk30} provided a generalization of Cantor's theorem using Kuratowski measure of non-compactness, $\alpha$:
$$\alpha(A) = \inf\left\{\varepsilon>0: \exists~X_i, i=1, \cdots, n, X_i\subset X, d(X_i)<\varepsilon, A=\cup_{i}X_i\right\},$$where $A$ is a subset of a complete metric space $X$. Kuratowski's theorem states that for each decreasing sequence $\{K_n\}$ of nonempty closed subsets of a complete metric space $X$ with $\lim\limits_{n\to\infty}\alpha(K_n)\to 0$, $\bigcap\limits_{n=1}^\infty K_n$ is nonempty and compact. Later, this theorem is further generalized by Horvath \cite{ch85} using the same measure of compactness $\alpha$.
 
An Atsuji space, which is more general than compact spaces, has the property that each continuous function on it is uniformly continuous. A metric space $X$ is said to be an Atsuji space if the set of limit points $X'$ is compact in $X$ and for each $\epsilon>0$, the complement of the set $N_\epsilon (X') :=\bigcup\limits_{x\in X'}B(x,\epsilon)$ in $X$ is uniformly discrete, where $B(x,\epsilon)$ denotes the open ball centered at $x$ and with radius $\epsilon$. For a metric space, the property of being an Atsuji space lies in
between the compactness and the completeness. A detailed study on Atsuji spaces can be found in \cite{tj07}.

This article presents various results on nonempty intersection of decreasing sequence of nonempty closed bounded subsets in metric spaces and in Atsuji spaces using Hausdorff distance $H(A, B)$ and the functional $\hat d$, defined as $\hat d(A) = \sup\limits_{x\in A} d(x, X\setminus A)$, where $A, B$ are subsets of a metric space $X$. These obtained results are also compared with the Cantor's intersection theorem.

The article is organized as follows. Some preliminary results, needed for the rest part of the article, are discussed in Section \ref{Pre}. In Section \ref{nip} and Section \ref{hat(dn)to0}, we consider decreasing sequence $\{K_n\}$ of nonempty closed bounded subsets of a metric space $X$ and discuss their nonempty intersection results in the cases for which the Hausdorff distance $H_n:=H(K_n,K_{n+1})\to 0$ and $\hat d(K_n):=\sup\limits_{x\in K_n}d(x,X\setminus K_n)\to 0$, respectively.
  	
\section{Preliminaries}\label{Pre}
Given a metric space $(X, d)$, we denote the set of all nonempty bounded subsets of $X$ and the set of all nonempty bounded closed subsets of $X$ by $B(X)$ and $C_b(X)$, respectively. Further given $A\subset X$, $A', \interior A$, $N_\epsilon (A)$ and $\partial A$ denote the set of all limit points, interior points, $\epsilon$-neighborhood and boundary of $A$, respectively. The diameter of $A$ is given by $\delta (A)= \sup\limits_{x,y\in A}d(x, y)$.
  	
  	\begin{defn}\cite{jh99} The {\it Hausdorff distance}, $H$, of two nonempty subsets $A,B $ of a metric space $(X,d)$ is defined as $$H(A,B)=\max\left\{\sup\limits_{x\in A}d(x,B), ~\sup\limits_{x\in B}d(x,A)\right\},$$ where $d(x,A)=\inf\limits_{y\in A}d(x,y)$.
  	\end{defn}
It is well-known that the distance function $H$ is a metric, provided $A,B$ are closed and bounded.
  	
\begin{defn}\label{acsms}\cite{hm38}
  		A sequence $\{x_n\}$ in a metric space $(X,d)$ is said to be \textit{absolutely convergent sequence} if $\sum_{i=1}^{\infty}d(x_i,x_{i+1})$ is finite.
 \end{defn}
\begin{lemma} \cite{hm38} \label{cscacs}
	In a metric space $X$, every Cauchy sequence $\{x_n\}$ contains an absolutely convergent subsequence.
\end{lemma}  	
 \begin{thm}\label{AtsujiI(x)}\cite{tj07}
In an Atsuji space $(X,d)$, every sequence $\{x_n\}$ with $\lim\limits_{n\to\infty} I (x_n)=0$ has a Cauchy subsequence, where $I(x)=d(x,X\setminus \{x\}), x\in X$.
\end{thm}

\begin{defn} \cite{am04}
	A subset $S$ of a metric space $X$ is said to be metrically convex if for any distinct points $x,y\in S$, there is a point $z\in S\setminus \{x,y\}$ that satisfies $d(x,y)=d(x,z)+d(z,y)$.
\end{defn}

\begin{thm}\cite{am04}\label{mspcunqcrv}
	Consider a complete and metrically convex metric space $X$. Then, for any distinct points $x,y\in X$, there is a metric segment with the end points $x,y$.
\end{thm}
 
\section{Nonempty intersection results using Hausdorff distance}\label{nip}
By a decreasing sequence $\{K_n\}$ of subsets of a metric space $X$, we mean $K_{n+1}\subset K_n, ~\forall~ n\in \mathbb N$, and we denote $H(K_n,K_{n+1})$ by $H_n$. Clearly $H_n\leq \delta(K_n)$. The following theorem furnishes a partial generalization of Cantor's intersection theorem, which we will discuss in Subsection \ref{nip-1}. It is well known that, in a metric space $X$, if for each decreasing sequence $\{F_n\}\subset C_b(X)$ with $\delta (F_n)\to 0$, $\bigcap\limits_{n=1}^{\infty}F_n\neq \emptyset$, then $X$ is complete.

\begin{thm}\label{itmc1}
  	A metric space $X$ is complete if and only if for every decreasing sequence $\{K_n\}\subset C_b(X)$ with $\sum_{n=1}^{\infty}H_n$ converges, $\bigcap\limits_{n=1}^{\infty}K_{n}\neq \emptyset$.
  \end{thm}
  \begin{proof} 
  	  		Let $X$ be a complete metric space.
  	  		For $a_1 \in K_1 $, $\epsilon > 0$, there exists $a_2\in K_2 $ such that
  	  		$d(a_1,a_2)\leq H(K_1, K_2)+\epsilon.$
  	  		Again, for $a_2\in K_2 $ and $\epsilon>0$ as above, there exists $a_3 \in K_3$ such that
  	  		$d(a_2,a_3)\leq H(K_2, K_3)+{\epsilon}^2.$
  	  		Proceeding this way, we get
  	  		$d(a_r,a_{r+1})\leq H(K_r, K_{r+1})+{\epsilon}^r,\ r\geq 1.$
  	  		This implies, $ \{a_i\}$ is a Cauchy sequence. Let $a_i\to a \in X$.  Then, $a \in \bigcap\limits_{i=1}^{\infty}K_n.$
  	
  	For the converse, consider a decreasing sequence $\{F_n\}\subset C_b(X)$ with $\delta (F_n)\to 0$. Then, the sequence $\{x_n\}\subset X$ with $x_n\in F_n$ is a Cauchy sequence. So, by Lemma \ref{cscacs}, it has an absolutely convergent subsequence, say, $\{x_{p_i}\}_{i=1}^{\infty}$. Let $K_i$ be the closure of the set $\{x_{p_i},x_{p_{i+1}},x_{p_{i+2}},...\},~ i=1,2,3,...\ .$ Observing $H_i\leq d(x_{p_i},x_{p_{i+1}})$, we have $\sum_{i=1}^{\infty}H_i<\infty$. Then, by the hypothesis, $\bigcap\limits_{i\in \mathbb N}K_{i} \neq \emptyset$. And hence, $\bigcap\limits_{n\in \mathbb N}F_n \neq \emptyset$. This completes the proof.
  \end{proof}
Following examples validate the statement of the above theorem.
\begin{example}\label{ex1}
	Let $K_n=\left[-\frac{1}{n},\frac{1}{n}\right]\subset \mathbb R$. Then $\sum_{n=1}^{\infty}H_n$ converges, and $\bigcap\limits_{n\in \mathbb N}K_n=\{0\}.$
\end{example}

\begin{example}\label{ex2}
	Consider the sequence space $X= (l^p, \|\cdot\|_p)$, for some $p$ with $1\leq p\leq \infty$ and choose $K_n=\{e_i\}_{i\geq n}$, where $e_i= (\delta_{ij})_{j=1}^\infty$. Then $\{K_n\}$ is a decreasing sequence in $C_b(X)$. It can be easily checked that $\sum_{n=1}^\infty H_n$ doesn't converge and $\bigcap \limits_{n=1}^{\infty} K_n=\emptyset$.
\end{example}	

\begin{example}\label{ex3}
Let $X=\mathbb Q$ with standard metric. Then $X$ is not complete. Let $r$ be a fixed irrational. Define $K_n=\{x\in X: r-1/n\leq x\leq r+1/n\}$. Then $K_n\in C_b(X)$ and decreasing. Although $\sum_{n=1}^\infty H_n$ converges, $\bigcap \limits_{n=1}^\infty K_n=\emptyset$.
\end{example}
 
 \subsection{Comparison with Cantor's theorem}\label{nip-1}
It is worth noting that Examples \ref{ex1}-\ref{ex3} also validate the result of Cantor's intersection theorem. Therefore it is fairly natural to ask: {\it What advantage Theorem \ref{itmc1} provides over the Cantor's theorem?} The answer lies in the fact that $H_n\leq \delta(K_n)$.

In Cantor's intersection theorem, $\delta(K_n)\to 0$ is the sufficient condition to have nonempty intersection. But in the case when $\delta(K_n)\not\to 0$, Cantor's theorem does not provide a conclusion whether $\bigcap\limits_{n=1}^{\infty}K_n$ is  empty or nonempty. In such case, if $\sum_{n=1}^{\infty}H_n$ converges, then Theorem \ref{itmc1} ensures $\bigcap\limits_{n=1}^{\infty}K_n$ is nonempty. For instance,

\begin{example}
	Let $K_n\subset \mathbb R^2$ be the region (including boundaries) bounded by the curves $4n(y-1/n)=-x^2$ and $4n(y+1/n)=x^2$, $n \in \mathbb N$. Then $\delta (K_n) \not \to 0$, and so Cantor's theorem becomes inconclusive here. However, $\sum_{n=1}^{\infty}H_n$ is convergent, and $\bigcap\limits_{n\in \mathbb N}K_n$ is the set $\{(x,0):-2\leq x\leq 2\}$.
\end{example}

 \subsection{Nonempty intersection in Atsuji space}
 
 In Theorem \ref{itmc1}, we see that the condition ``$\sum_{n=1}^\infty H_n<\infty$'' is sufficient to have $\bigcap\limits_{n\in \mathbb N}K_n\neq \emptyset$ in complete metric spaces. A more general condition, namely ``$H_n\to 0$'', is not sufficient to have the nonempty intersection. For instance,
  \begin{example}
  	The functions $e_i$, defined as $e_i(t)=t^i, t\in [0,1], i\in \mathbb N$, are in the normed space $X=(C[0,1], \|\cdot\|_{\infty})$, and the sets $K_n:=\{e_i\}_{i=n}^{\infty},\ n\in \mathbb N$, are closed bounded subsets of $X$. Here, $H_n=\|x_n-x_{n+1}\|_{\infty}\to 0$, but $\bigcap\limits_{n=1}^{\infty}K_n=\emptyset.$
  \end{example}
  However, in Atsuji spaces, which are also complete, ``$H_n\to 0$'' is sufficient to have the nonempty intersection. 
\begin{thm}\label{IntrsctnInAtsjiSpc}
	If $X$ is an Atsuji space, then for each decreasing sequence $\{K_n\}\subset C_b(X)$ with $H_n\to 0$, $\bigcap\limits_{n=1}^{\infty}K_n\neq \emptyset.$
\end{thm}
 
 \begin{proof}
	Given a decreasing sequence $\{K_n\}\subset C_b(X)$ such that $H_n\to 0$. For the sequence $\{K_n\}$, there exists $a_n\in K_n$ (as in the proof of Theorem \ref{itmc1}) such that for any fixed $\epsilon \in (0,1)$ and for all $n\geq 1$ we have, 
	$d(a_n,a_{n+1})\leq H_n+\epsilon^n \to 0$ as $n\to \infty.$ If $a_n=c$ for infinitely many values of $n$, then $c\in \bigcap\limits_{n=1}^\infty K_n$. Otherwise, if the terms of the sequence $\{a_n\}$ are distinct, except for at most finitely many $n$, then this implies $I(a_n)\to 0$. Therefore, by Theorem \ref{AtsujiI(x)}, $\{a_n\} $ has a Cauchy subsequence, and hence the limit of this sequence is in $\bigcap\limits_{n=1}^\infty K_n.$ 
\end{proof}
 
 The converse of Theorem \ref{IntrsctnInAtsjiSpc}, in general, is not true.  That means if ``$H_n\to 0$'' suffices to have the nonempty intersection property, the space is not necessarily an Atsuji space. This is evident from the following example.

%

\begin{example}\label{CnvrsHnAtsuji}
Consider $X=\mathbb N \cup M$ with the standard Euclidean metric on $\mathbb R$, where $M=\{n+1/2m:m,n\in \mathbb N\}$. Let $\{K_i\}\subset C_b(X)$ be a decreasing sequence with $H_i\to 0$. Then, as in the proof of Theorem \ref{IntrsctnInAtsjiSpc}, there exists $a_i\in K_i$ such that $|a_i-a_{i+1}|\to 0$. It can be shown that either $\{a_i\}$ is eventually constant, say $p\in X$, or $\{a_i\}$ converges to some positive integer $k$. In either case $\bigcap\limits_{i=1}^{\infty}K_i\neq \emptyset$, as it contains either $p$ or $k$. 

Therefore, for each decreasing sequence $\{K_i\}\subset C_b(X)$ with $H_i\to 0$, $\bigcap\limits_{i=1}^{\infty}K_i\neq \emptyset$. But the space $X$ is not an Atsuji space, because $X'=\mathbb N$.
\end{example}

 We notice that Theorem \ref{IntrsctnInAtsjiSpc} provides a generalization of Cantor's intersection theorem in Atsuji spaces.
 
 \section{Nonempty intersection results using $\hat d$}\label{hat(dn)to0}
Jain and Kundu, in \cite{tj07}, considered a functional $I:X\to \mathbb R$ defined as, $I(x)=d(x,X\setminus \{x\})$. Here we consider a more general functional, $\hat d$, acting on the subsets of a  metric space $X$, defined by $\hat d(A)=\sup\limits_{x\in A}d(x,X\setminus A)$, $A\subset X$. This is also known as the excess of $A$ over $X\setminus A$. This functional gives the radius of the inscribed ball inside a regular body in the Euclidean space $\mathbb R^2$. For subsets $A,B$ of a metric space $X$, we conclude the following as well: 
  \begin{enumerate}
  	\item $\hat d(A)\leq \hat d(B)$ if $A\subset B$, and $A\in B(X)$,
  	\item $\hat d(A)=0 $ if and only if $A=\emptyset$ or $A\subset [X\setminus A]'$,
  	\item  For an unbounded set $A$, $\hat d(A)$ can be finite or infinite. For example, in the euclidean space $\mathbb R^2$, consider $A_1=\{(x,0):x\in \mathbb R\}$ and $A_2=\{(x,y):x,y\geq 0\}$, then $\hat d(A_1)=0$ and $\hat d(A_2)=\infty$,
  	\item If $A=X$, then $\hat d(A)$ is finite or infinite, depending on $X$ is bounded or unbounded, respectively. (We take $d(x,\emptyset)= \sup\{d(x,A):A\subset X\},~ x\in X$.)
  \end{enumerate}

\begin{lemma}\label{H(X-A,X-B)}
  	Let $A,B\in B(X)$. Then
  	\begin{enumerate}
  		\item[(a)] $H(X\setminus A,X\setminus B)\leq \max\{\hat d(A),\hat d(B)\}$.
  		\item[(b)] $ 
  		\max\{\hat d(A),\hat d(B)\}\leq H(A,B)$, provided $A\cap B = \emptyset$.
  	\end{enumerate}
  \end{lemma}
  \begin{proof} (a) For $A,B\in B(X)$, we have
  	\begin{align}
  	&H(X\setminus A, X\setminus B)\nonumber\\
  	&=\max\{\sup\limits_{p\in X\setminus A}d(p,X\setminus B),\sup\limits_{q\in X\setminus B}d(X\setminus A,q)\}\nonumber\\
  	&=\max \{\sup\limits_{p\in B\setminus A}d(p,X\setminus B),\sup\limits_{q\in A\setminus B}d(q,X\setminus A)\}\label{eqn1}\\
  	&\leq \max \{\sup\limits_{p\in B}d(p,X\setminus B),\sup\limits_{q\in A}d(q,X\setminus A)\}\nonumber\\
  	&= \max\{\hat d(B),\hat d(A)\}. \label{eqn3}
  	\end{align}
  	(b) Suppose $A\cap B=\emptyset$, then
  	\begin{align*}
  	H(A,B)
  	&= \max \{\sup\limits_{x\in A}d(x,B),\sup\limits_{y\in B}d(A,y)\}\\
  	&\geq \max \{\sup\limits_{x\in A}d(x,X\setminus A),\sup\limits_{y\in B}d(X\setminus B,y)\}  \\
  	&=\max\{\hat d(A),\hat d(B)\}.
  	\end{align*}
  \end{proof}	

\begin{rem}
	For two subsets $A,B$ with $A\subset B$ in a metric space $X$, in general, $\hat d(B)$ and $H(A,B)$ are not comparable, even in Atsuji spaces.
\end{rem}

\begin{thm}\label{HatdAtsuji}
	Let $X$ be an Atsuji space. Then, for each decreasing sequence $\{K_n\}\subset C_b(X)$ with $\hat d(K_n)\to 0$, $ \bigcap\limits_{n=1}^{\infty} K_n\neq \emptyset$.
\end{thm}
\begin{proof}
	For $x\in X$ and for $A, B \subset X$, we have $d(x,A) \leq d(x,B)+H(A, B)$. Since $K_{n+1}\subset K_n$, using (a) of Lemma \ref{H(X-A,X-B)} we get, $H(X\setminus K_n,X\setminus K_{n+1})\leq \hat d(K_n)$. Now, consider $x_n\in K_n\setminus K_{n+1}$. Then
	$I(x_n)=d(x_n,X\setminus \{x_n\})\leq d(x_n,X\setminus K_n)\leq d(x_n,X\setminus K_{n+1})+H(X\setminus K_n, X\setminus K_{n+1})$ $\leq 0+ \hat d(K_n) \to 0$ as $n\to \infty$. So using Theorem \ref{AtsujiI(x)}, the sequence $\{x_n\}$ has a Cauchy subsequence. Therefore its limit is in $\bigcap\limits_{i=1}^{\infty}K_n.$
\end{proof}

The converse of Theorem \ref{HatdAtsuji}, in general, does not hold.
\begin{example}
		Consider the set $X$ as in Example \ref{CnvrsHnAtsuji}, with the standard Euclidean metric $d$ on $\mathbb R$. Let $\{K_i\}\subset C_b(X)$ be a decreasing sequence with $\hat d(K_i)\to 0$. Then, as in the proof of Theorem \ref{HatdAtsuji}, there exists $x_i\in K_i\setminus K_{i+1}$ such that $I(x_i)\to 0$. If the range set $R=\{x_i\}_{i\geq 1}$ is infinite and $R\subset M$, consider $x_i=n_i+1/2m_i$. If $\{m_i\}_{i\geq 1}=\{p_1,p_2,...,p_q\}$, a finite set, then $d(n_i+1/2p_j, X\setminus \{n_i+1/2p_j\})= d(n_i+1/2p_j,n_i+1/2(p_j+1))=\frac{1}{2}d(1/p_j,1/(p_j+1))\geq \inf\limits_{1\leq j\leq q}\frac{1}{2}d(1/p_j,1/(p_j+1))>0$, which is a contradiction to $I(x_i)\to 0$. Hence, there is a subsequence $\{m_{t_i}\}$ of $\{m_i\}$ with $m_{t_i}\to \infty$ as $i\to \infty$. And so, it can be proved that in the case either $ R\subset M$ or $R\not\subset M$; the sequence $\{x_i\}$ has a subsequence converging to some $p\in \mathbb N$. Therefore $p\in\bigcap\limits_{i=1}^{\infty}K_i$.
		
		
		Thus we get, for each decreasing sequence $\{K_i\}\subset C_b(X)$ with $\hat d(K_i)\to 0$, $\bigcap\limits_{i=1}^{\infty}K_i\neq \emptyset$. Although, the space $X$ is not an Atsuji space.
	\end{example}

	\begin{thm} 
	If $X$ is a metric space, and for each decreasing sequence $\{K_n\}\subset C_b(X)$ with $\hat d(K_n)\to 0, \ \bigcap\limits_{n=1}^{\infty} K_n\neq \emptyset$, then $X$ is complete.
\end{thm}
\begin{proof}
	Consider a decreasing sequence $\{F_n\}\subset C_b(X)$ with $\delta(F_n)\to 0$. Then, the sequence $\{x_n\}$ with $x_n\in F_n$ is a Cauchy sequence in $X$.  Let $K_n$ be the closure of the set $\{x_i\}_{i\geq n}$. Since $\{x_n\}$ is Cauchy, for each $\epsilon>0$, there is an $N\in \mathbb N$ such that 
	 $\sup\limits_{n\geq N+1}d(x_N,x_n)<\epsilon$, that is, $\sup\limits_{x\in K_{N+1}}d(x_N,x)<\epsilon$, which further implies $\hat d(K_{N+1})=\sup\limits_{x\in K_{N+1}}d(x, X\setminus K_{N+1})<\epsilon$. Thus $\hat d(K_n)\to 0$. So, by the hypothesis $\bigcap\limits_{n=1}^{\infty} K_n\neq \emptyset$. Hence, $\bigcap\limits_{n=1}^{\infty} F_n\neq \emptyset$ and this completes the proof.
\end{proof}

\subsection{Comparison with Cantor's theorem}
We observe that, in general, $\hat d$ and $\delta$ are not comparable. For instance,
\begin{example}
		Consider a set $X=\{x,y\},~x\neq y$, equipped with a metric $d$. Let $A=\{x\}$. Then, $\hat d(A)=d(x,y)>0$ and $\delta(A)=0$.
		
		But if $X=\mathbb R^2$ with standard Euclidean metric and $A=B[0,r]$, then $\hat{d}(A)=r<2r=\delta(A)$. 
	\end{example}
We shall show that in metrically convex metric spaces, $\hat d$ is always dominated by $\delta$.

\begin{lemma}\label{BcapScap}
	Let $(X,d)$ be a metric space. Then, for all $x,y\in X$ and $r\in [0,d(x,y)]$,  $B[x,r]\cap B[y,d(x,y)-r]=S[x,r]\cap S[y,d(x,y)-r]$, where $S[x,r]:=\{z\in X:d(x,z)=r\}$.
\end{lemma}
\begin{proof}
	Let us denote $B[x,r]\cap B[y,d(x,y)-r]$ and $S[x,r]\cap S[y,d(x,y)-r]$ by $B^{\cap}$ and $S^{\cap}$, respectively. If $B^{\cap}$ is empty, there is nothing to prove. Let $B^{\cap}\neq \emptyset$ and $z \in B^{\cap}$. We claim $z \not\in B_1\cup B_2$, where $B_1=B[x,r]\setminus S[x,r]$, $B_2=B[y,d(x,y)-r]\setminus S[y,d(x,y)-r]$. If possible, let $z\in B_1$. Then $d(x,y)\leq d(x,z)+d(z,y)<r+d(x,y)-r=d(x,y)$, which is a contradiction. Similarly, we prove $z\not\in B_2$. Thus, $z\in B^{\cap}\cap [B_1\cup B_2]^c=B^{\cap}\cap [{B_1}^c\cap {B_2}^c]$, where $B^c$ denotes the complement of a set $B$ in $X$. This implies $z \in S^{\cap}$, and so $B^{\cap}\subset S^{\cap}$. Hence, $B^{\cap} =S^{\cap}$.
\end{proof}

\begin{thm}\label{hat,delta}
	Let $A$ be a nonempty bounded proper subset in a complete metrically convex space $X$. Then $\hat d(A)\leq \delta(A)$.
\end{thm}

\begin{proof}
	Let $(X,d)$ be an metrically convex space. By Theorem \ref{mspcunqcrv}, $X$ is a connected metric space too. We shall achieve the conclusion in the following steps.
	
	{\bf Step 1.} \underline{$N_\epsilon(A)\setminus A$ is not empty for each $\epsilon>0$.} \\On the contrary, let us assume $N_\epsilon (A)\setminus A=\emptyset$. This implies, $A$ is open and for all $x\in A$, $B(x,\epsilon)\subset A$. Since $A$ is open,  $ \partial A\not\subset A$. This implies, ``$\exists~ b\in \partial A$ such that $b\not\in B(x,\epsilon)$ for all $x\in A$'', which is a contradictory statement in itself.
	
	{\bf Step 2.} \underline{$B(a,r+\epsilon) = N_\epsilon(B)$ for a given $\epsilon>0$, where $B=B(a,r)$.} \\
	Observe that, $N_\epsilon(B)\subset B(a, r+\epsilon)$ follows from the triangle inequality. Conversely, suppose $y\in B(a,r+\epsilon)$. If $y\in B$ then $y\in B(y, \epsilon) \subset N_\epsilon(B)$. Suppose $y\in B(a,r+\epsilon)\setminus B$. Then, consider the ball $B(y,\epsilon)$. If $B(y,\epsilon)\cap B=\emptyset$, then $B(y, \epsilon)\cap \partial B=\emptyset$. Therefore, $d(q,y)\geq \epsilon$ for all $q\in \partial B$. By Theorem \ref{mspcunqcrv} and Lemma \ref{BcapScap}, since $r\in [0,d(a,y)]$, there is a $u\in X$ such that $\psi (r)=u \in $ $B[a,r]\cap B[y,d(a,y)-r]=S[a,r]\cap S[y,d(a,y)-r]=\partial B\cap S[y,d(a,y)-r]$, where $\psi$ is an isometry from $[0,d(a,y)]$ to $X$.  This implies $d(a,y)=d(a,u)+d(u,y)\geq r+\epsilon$, a contradiction. Hence, $B(y,\epsilon)\cap B\neq\emptyset$, and therefore $y\in B(z,\epsilon)$ for some $z\in B$. Thus, $B(a,r+\epsilon)\subset N_\epsilon(B)$.

{ \bf Step 3.} \underline{ $\hat d(A)\leq \delta(A)$.}\\
	Since $\partial B\neq \emptyset$, $\inf\limits_{p\in X\setminus B}d(a,p)=d(a,q)=r$, for some $q\in \partial B$. Now, let $x$ be an element in $A$. If $x\in \interior{A}$, then consider a ball $B(x,r')$, where $r'=\sup\{r>0:B(x,r)\subset A\}$. There must be a point in $N_\epsilon(B(x,r'))$ which lies in $N_\epsilon(A)\setminus A$. Hence, $\big[N_\epsilon(B(x,r'))\setminus B(x,r')\big]\cap \big[N_\epsilon (A)\setminus A\big]\neq \emptyset$. Therefore, $d(x,X\setminus A)=r'\leq \delta (A)$. On the other hand, if $x\in \partial A$, then $d(x,X\setminus A)=0\leq \delta(A)$. Hence, $\hat d(A)\leq \delta(A)$.
\end{proof}
	

    Since the metric spaces with Takahashi's convex strustures (\cite{wp14}) and the normed spaces are metrically convex metric spaces, so Theorem \ref{hat,delta} is applicable to these spaces too.	
	Due to Theorem \ref{hat,delta}, Theorem \ref{HatdAtsuji} induces a generalization of Cantor's intersection theorem in metrically convex Atsuji spaces.
	
	\begin{example}
		Consider the metric space $X=\{(x,y):-3\leq x,y\leq 3\}$ with standard Euclidean metric on $\mathbb R^2$. The space $X$ is a metrically convex Atsuji space. Let $K_n\subset X$ be the region (including boundaries) bounded by the curves $n(1+\frac{1}{n})^2(y-1/n)=-x^2$ and $n(1+\frac{1}{n})^2(y+1/n)=x^2$, $n \in \mathbb N$. Here $\delta (K_n) \not \to 0$, and so Cantor's theorem becomes indecisive here. However, $\hat d(K_n)\to 0$, and $\bigcap\limits_{n\in \mathbb N}K_n$ is the set $\{(x,0):-1\leq x\leq 1\}$.
	\end{example}


\end{document}